\documentclass[a4paper,12pt,intlimits,oneside]{amsart}

\usepackage{amsmath}
\usepackage{amssymb}
\usepackage{amsfonts}
\usepackage{amsthm,amscd}

\newtheorem{thm}{Theorem}[section]
\newtheorem{prop}[thm]{Proposition}

\newtheorem{lem}[thm]{Lemma}
\newtheorem{defn}[thm]{Definition}

\theoremstyle{definition}
\newcommand{\comment}[1]{}

\numberwithin{equation}{section}

\theoremstyle{definition}

\begin{document}
\title[ Hardy-Fofana spaces and temperature equation]{New characterization of Hardy-Fofana spaces and temperature equation}
\author[M. A. Dakoury]{Martial Agbly Dakoury}
\address{Laboratoire de Math\'ematiques et Applications, UFR Math\'ematiques et Informatique, Universit\'e F\'elix Houphou\"et-Boigny Abidjan-Cocody, 22 B.P 582 Abidjan 22. C\^ote d'Ivoire}
\email{{\tt dakourymartial@gmail.com}}
\author[J. Feuto]{Justin Feuto}
\address{Laboratoire de Math\'ematiques et Applications, UFR Math\'ematiques et Informatique, Universit\'e F\'elix Houphou\"et-Boigny Abidjan-Cocody, 22 B.P 1194 Abidjan 22. C\^ote d'Ivoire}
\email{{\tt justfeuto@yahoo.fr}}


\date{}
%\comment{
\begin{abstract}
The aim of this paper is  to give a characterization of   Hardy-Fofana spaces  via Riesz trasforms. This characterization allow us to  describ the distributions belonging to these spaces  as a bounded solutions of Cauchy-Riemann's general temperature  equations.
\end{abstract}

\maketitle

\section{Introduction}
\label{sec:1}
Let $\mathbb R^d$ ($d$ is a positive integer) be the Euclidean space of dimension $d$  equipped  with the Lebesgue measure $dx$ and the Euclidean norm.  The classical Hardy space $\mathcal H^p(\mathbb R^d)$ $(0<p<\infty)$ is defined as the space of tempered distributions $f$ satisfying 
$\Vert \mathcal Mf\Vert_p <\infty,$ where the maximal function $\mathcal Mf$ is defined by
\begin{equation}
\mathcal Mf(x)=\sup_{t>0}\vert (f\ast \varphi_t)(x)\vert,\label{maxf}
\end{equation}
with $\varphi$ in the Schwartz class $\mathcal S(\mathbb R^d)$ having non vanish integral, and $\varphi_t(x)=t^{-d}\varphi(t^{-1}x)$. 

It is well  known that not only  this  space does not depends on $\varphi$, but one can replaced  Schwartz function by  Poisson kernel   in the definition of the maximal function (\ref{maxf}).% We can also characterize this space with "simple functions" known as atoms (see Section \ref{prerequisite} for the definition).

In \cite{AbFe}, Ablé and the second author studied  Hardy-amalgam spaces  $\mathcal H^{(p,q)}(\mathbb R^d)$ ($0<p,q<\infty$) by taking in the above maximal characterization of classical Hardy space  the Wiener amalgam quasi-norm $\left\|\cdot\right\|_{p,q}$ instead of Lebesgue's. 

 A locally integrable function $u$ belongs to the amalgam space $(L^p,\ell^q)(\mathbb R^d)$ if %$the function $y\mapsto\left\|f\chi_{B(y,1)}\right\|_{p}$ is in $L^q$, and put
$$\left\|u\right\|_{p,q}:=\left[ \sum_{k\in\mathbb Z^d}\Vert u\chi_{Q_k}\Vert_{p}^q\right]^{\frac{1}{q}}<\infty,$$
where for $k\in\mathbb Z^d$, $Q_k=k+\left[ 0,1\right) ^d$ and $\chi_{Q_k}$ stands for the characteristic function of $Q_k$. 

Multiple characterizations of $ \mathcal H^ {(p,q)} (\mathbb R^d) $ spaces including atomic and Poisson kernel characterization, were given in \cite{AbFe}. We notice that the atoms in this context are exactely the one used in classical Hardy space. 

Recently, Assaubay et al in \cite{AjCf} characterized this spaces by using first-order classical Riesz transforms and composition of first-order Riesz transformations. They also describe the distributions in $\mathcal H^{(p,q)}(\mathbb R^d)$ as the boundary values of solutions of harmonic and caloric Cauchy-Riemann systems. Here we intend to prove that similar characterizations are possible in the context of Hardy-Fofana spaces.

%The authors considered in \cite{DmFe} some particular subspaces of Hardy-amalgam spaces, namely the Hardy-Fofana spaces.

It is well known that 
%Wiener amalgam spaces are stable with respect to dilations. More precisely
 for $0<p,\alpha,q<\infty$ and $r>0$, there exists a constant $C_{r;\alpha}>0$ depending on $r$ and $\alpha$ such that  \begin{equation}
 C_{r;\alpha}^{-1}\Vert u\Vert_{p,q}\leq \Vert St^{\alpha}_{r} u\Vert_{q,p}\leq C_{r;\alpha}\Vert u\Vert_{p,q},\qquad u\in (L^p,\ell^q)(\mathbb R^d),\label{contr2}
 \end{equation}
  where  
 $(St^{\alpha}_{r} u)(x)= r^{-\frac{d}{\alpha}} u (r^{-1} x)$.  It follows from the above relation that for $u \in (L^p, \ell^q)(\mathbb R^d)$, we have $\mathrm{St}^\alpha_r u \in (L^p, \ell^q)(\mathbb R^d)$ for $\alpha > 0$ and $r > 0$. Unfortunately, the family $\left\lbrace \mathrm{St}^\alpha_ru\right\rbrace _{r > 0}$ is not bounded in $(L^p, \ell ^q)(\mathbb R^d)$.
Ibrahim Fofana considered in \cite{Fo1}, the spaces $(L^p,\ell^q)^\alpha(\mathbb R^d)$ defined for $0<p,q,\alpha\leq\infty$ by
$$(L^p,\ell^q)^\alpha(\mathbb R^d)=\left\{f\in(L^p,\ell^q)(\mathbb R^d)/\left\|f\right\|_{p,q,\alpha}<\infty\right\}$$
where 
\begin{equation}
\left\|f\right\|_{p,q,\alpha}:=\sup_{r>0}\left\|St^{\alpha}_{r}f\right\|_{p,q}.\label{normqpa}
\end{equation} 
 These spaces known as  Fofana's spaces  are non trivial if and only if $p\leq \alpha\leq q$ (see \cite{Fo1}). In the rest of the paper  we will always assume that this condition is fulfilled. It is proved in \cite{Fe1} that for $u\in(L^p,\ell^q)^\alpha(\mathbb R^d)$, we have $\Vert St^{\alpha}_ru\Vert_{p,q,\alpha}=\Vert u\Vert_{p,q,\alpha}$ and that $(L^p,\ell^q)^\alpha(\mathbb R^d)$  ($1\leq p\leq \alpha \leq q$) is the biggest norm space which is continuously embedded in $(L^p,\ell^q)(\mathbb R^d)$ and for which the translation $\mathrm{St}^\alpha_r$ is an isometry. %$\sup_{r>0}\Vert St^{\alpha}_r u\Vert_{p,q} <\infty$.
 These spaces can also be viewed as some generalized Morrey spaces since for $p <\alpha$, the space $(L^p,\ell^\infty)^\alpha(\mathbb R^d)$ is exactly the classical
Morrey space $L^{p, d\frac{p}{\alpha}}(\mathbb R^d)$. 

For $0<p\leq \alpha\leq q<\infty$, Hardy-Fofana space $\mathcal H^{(p,q,\alpha)}(\mathbb R^d)$, introduced by the authors in \cite{DmFe}, is a subspace of Hardy-amalgam spaces consists of tempered distributions $f$ satisfying 
$$\Vert f\Vert_{\mathcal H^{(p,q,\alpha)}}:=\Vert \mathcal M f\Vert_{p,q,\alpha}<\infty.$$
The purpose of this article is twofold. We first  characterize these spaces via Riesz transforms  and secondly, we describe the distributions belonging to these spaces as bounded solutions of certain general temperature equations of CAUCHY-RIEMANN.

This paper is organized as follow:

The next Section is devoted to the prerequisites on Hardy-Fofana spaces. In Section 3, we give the characterizations of Hardy-Fofana spaces with Riesz transforms. In the last section, we characterize distributions belonging to our spaces as bounded solutions of certain general temperature equations of CAUCHY-RIEMANN.

In this work, $\mathcal S := \mathcal S(\mathbb R^{d})$ will denote the Schwartz class of rapidly decreasing smooth functions equipped with its usual topology. The dual space of $\mathcal S$ is the space of tempered distributions denoted by $\mathcal S' := \mathcal S'(\mathbb R^{d})$. The pairing between $\mathcal S'$ and $\mathcal S$ is denoted by $\left\langle \cdot,\cdot\right\rangle$. %The notation $\mathbb N^\ast$ will stands for the set of positive integers.

We denote by  $\left|E\right|$, the Lebesgue measure of a measurable subset $E$ of $\mathbb R^d$.
The notation $A\approx B$ means that there exist two constants  $0<C_{1}$ and $0<C_{2}$ such that $A\leq C_{1} B$ and  $B\leq C_{2}A$, while $A:= B$ means that $B$ is the definition of $A$.

\section{Prerequisites for Hardy-Fofana spaces}\label{prerequisite}
Fofana's spaces have among others, the following properties (see for example \cite{Fe1} and \cite{Fo1}):
\begin{enumerate} 
\item let $0<p,\alpha,q\leq \infty$. The space $\left((L^p,\ell^q)^\alpha(\mathbb R^d),\left\|\cdot\right\|_{p,q,\alpha}\right)$ %is non trivial if and only if $p\leq\alpha\leq q$, 
is a Banach space if $1\leq p\leq \alpha\leq q$ and a quasi-Banach space if $0<p<1$;\label{rfof1}
\item if $\alpha\in \left\{p,q\right\}$ then $(L^p,\ell^q)^\alpha(\mathbb R^d)=L^\alpha(\mathbb R^d)$  with equivalent norms;\label{rfof2}
\item if $p<\alpha<q$ then $L^\alpha(\mathbb R^d)\subsetneq (L^p,\ell^q)^\alpha(\mathbb R^d)\subsetneq (L^p,\ell^q)(\mathbb R^d)$; 
\item let $f$ and $g$ be two measurable functions on $\mathbb R^d$. If 
$\vert f\vert\leq \vert g\vert$, then $\Vert f\Vert_{p,q,\alpha}\leq \Vert g\Vert_{p,q,\alpha}$.\label{rfof4}
\end{enumerate}

For many operators including the maximal Hardy-Littlewood operator, norm inequalities are given in these spaces for $1\leq p\leq \alpha\leq q$.

Let $f$ be a locally integrable function and  $\mathfrak{M}(f)$ be the centered  Hardy-Littlewood maximal function defined by 
 $$\mathfrak{M}(f)(x):=\sup_{r>0}|B(x,r)|^{-1}\int_{B(x,r)}|f(y)|dy,\ \forall\ x\in\mathbb{R}^{d}.$$ 
It is proved in \cite[Proposition 4.2]{Fe1} that $\mathfrak{M}$ is bounded on $(L^p,\ell^q)^\alpha(\mathbb R^d)$, whenever $1<p\leq\alpha\leq q\leq\infty$. 
Using \cite[Proposition 11.12]{LSUYY}, it is easy to extablish the following result whose proof is omitted.

\begin{prop}\label{operamaxalvec}
Let $1<p\leq\alpha\leq q<+\infty$ and $1<u\leq+\infty$. For all sequences  $\left\{f_n\right\}_{n\geq0}$ of measurable functions, we have 
$$\left\|\left(\sum_{n\geq0}|\mathfrak{M}(f_n)|^{u}\right)^{\frac{1}{u}}\right\|_{p,q,\alpha }\approx\left\|\left(\sum_{n\geq0}|f_n|^u\right)^{\frac{1}{u}}\right\|_{p,q,\alpha },$$ 
with the equivalence constants not depending on the sequence  $\left\{f_n\right\}_{n\geq0}$.
\end{prop}

As Hardy-Fofana spaces are concerned, we have among others, the following properties which can be found in \cite{DmFe}. 
\begin{prop}Let $1 \leq p\leq\alpha\leq q < \infty$.
\begin{enumerate}
\item If $1<p$ then the space $\mathcal H^{(p,q,\alpha)}(\mathbb R^d)$ and $(L^{p},L^{q})^{\alpha}(\mathbb R^d)$ are equal with equivalence norms.
\item The space $\mathcal H^{(1,q,\alpha)}(\mathbb R^d)$ is continuously embedded in $(L^1,\ell^q)^{\alpha}(\mathbb R^d)$.
\end{enumerate}
\end{prop}
Notice that for $p<1$, we have as in the  classical Hardy and Hardy-amalgam spaces, that the spaces $\left( \mathcal H^{(p,q,\alpha)}(\mathbb R^d),\Vert\cdot\Vert_{\mathcal H^{(p,q,\alpha)}}\right) $ are quasi-Banach and for $f ,g \in\mathcal H^{(p,q,\alpha)}(\mathbb R^d)$,
$$ \left\|f +g\right\|^{p}_{\mathcal H^{(p,q,\alpha)}}\leq \left\|f\right\|^{p}_{\mathcal H^{(p,q,\alpha)}}+\left\|g\right\|^{p}_{\mathcal H^{(p,q,\alpha)}}.$$  
 We can also define (see \cite{DmF2}) these spaces as subspaces of Hardy-amalgam spaces for which the familly of dilations $\left\lbrace \mathrm{St}^\alpha_\rho\right\rbrace _{\rho>0}$ is locally bounded. %which are uniformly bounded for $\left\lbrace \mathrm{St}^\alpha_\rho\right\rbrace _{\rho>0}$.
 
 More precisely, for a tempered distribution $f$, $\rho>0$ and $\alpha$ two  real numbers we put 
$$\left\langle \mathrm{St}^\alpha_\rho f,\varphi\right\rangle :=\left\langle f,\mathrm{St}^{\alpha'}_{\rho^{-1}}\varphi\right\rangle ,$$
where %the right hand side is defined as in the introduction and 
$\frac{1}{\alpha'}+\frac{1}{\alpha}=1$. We have (see \cite{DmFe}) that %e following result.
%\begin{prop} \label{1201}Let
for  $0 < p\leq\alpha\leq q\leq\infty$,  
%\begin{enumerate}
%\item For $0<p\leq \alpha\leq q<\infty$, we have 
\begin{equation}
\Vert f\Vert_{\mathcal H^{(p,q,\alpha)}}=\sup_{\rho>0}\Vert \mathrm{St}^\alpha_\rho f\Vert_{\mathcal H^{(p,q)}}.\label{deltachar}
\end{equation}
%where $\left\langle \mathrm{St}^\alpha_\rho f,\varphi\right\rangle :=\left\langle f,\mathrm{St}^{\alpha'}_{\rho^{-1}}\varphi\right\rangle $.
%Notice that  if $p\leq 1$ then  
Just as Hardy-amalgam spaces was characterized in \cite{AbFe} with Poisson kernel, so are Hardy-Fofana's spaces. In fact, a tempered distribution $f$  belonging to Hardy-amalgam spaces is bounded; i.e  $f\ast \psi\in L^{\infty}(\mathbb{R}^{d})$ for all $\psi \in \mathcal{S}(\mathbb{R}^{d})$. A convolution of such distribution with integrable functions can be defined in term of distribution. More precisely, if $f\in\mathcal S'(\mathbb R^d)$ is bounded  and $u\in L^1(\mathbb R^d)$, then  the convolution $f\ast u$  is  defined as a tempered distribution acting on $\mathcal S(\mathbb R^d)$ by the pairing
 $$\left\langle f\ast u,\varphi\right\rangle :=\left\langle f\ast \tilde{\varphi},\tilde{u}\right\rangle_{(L^\infty,L^1)}\qquad\varphi\in\mathcal S(\mathbb R^d)$$
where $\tilde{u}(x)=u(-x)$ and $\left\langle f\ast \tilde{\varphi},\tilde{u}\right\rangle_{(L^\infty,L^1)}$ is the pairing between $L^\infty(\mathbb R^d)$ and $L^1(\mathbb R^d)$. But if we take as $u$ the Poisson kernel $P$
defined by
\begin{eqnarray*}
P(x):=\frac{\Gamma(\frac{d+1}{2})}{\pi^{\frac{d+1}{2}}}\frac{1}{(1+|x|^{2})^{\frac{d+1}{2}}}\;\;\;   x\in\mathbb{R}^{d},
\end{eqnarray*}
then $f\ast P_t$ can be identified for all $t>0$, to a well defined bounded function. As we can see for example in \cite{Graf},  there exist $\varphi,\psi\in\mathcal S(\mathbb R^d)$ such that $$f\ast P_t=(f\ast \varphi_t)\ast P_t+f\ast \psi_t\text{ for }t>0.$$
% We recall that  the Poisson kernel $P$ is  
It is proved in \cite{AbFe}  that for an element $f\in\mathcal H^{(p,q)}(\mathbb R^d)$, we have 
\begin{equation}
\Vert x\mapsto\sup_{t>0}\sup_{\vert x-y\vert<t}\vert f\ast P_t(y)\vert\Vert_{p,q}\approx\Vert\mathcal Mf\Vert_{p,q}\label{Poissonhpq}
\end{equation}
where $\mathcal M f$ is the maximal function defined in Relation (\ref{maxf}). It follows that 
\begin{equation}
\Vert x\mapsto\sup_{t>0}\sup_{\vert x-y\vert<t}\vert f\ast P_t(y)\vert\Vert_{p,q,\alpha}\approx\Vert\mathcal Mf\Vert_{p,q,\alpha},
\end{equation}
thanks to Relations (\ref{deltachar}) and (\ref{Poissonhpq}) and the fact that $\mathrm{St}^\alpha_\rho$ commute with the maximal function $\mathcal M$.
% We have the following lemmas.
 \begin{lem}\label{dilconv}
 Let $f\in\mathcal S'(\mathbb R^d)$, $\varphi\in\mathcal S(\mathbb R^d)$, $\rho$ and $\alpha$ positive real numbers. We have 
 $$\mathrm{St}^\alpha_\rho\left( f\ast\varphi_t\right) =\left( \mathrm{St}^\alpha_\rho f\right) \ast\varphi_{\rho t},\quad t>0.$$
 \end{lem}
 Infact, 
 \begin{eqnarray*}
\rho^{\frac{-d}{\alpha}}\left( f\ast \varphi_{t} \right)  (\rho^{-1}x)&=&\rho^{\frac{-d}{\alpha}}\left\langle f, \rho^d\varphi_{\rho t}( x-\rho \cdot)\right\rangle \\
&=&\rho^{\frac{d}{\alpha'}}\left\langle f, \varphi_{\rho t}( x-\rho\cdot)\right\rangle\\
&=&\left\langle f, \mathrm{St}^{\alpha'}_{\rho^{-1}}\left[  \varphi_{\rho t}( x-\cdot)\right] \right\rangle=\left( \mathrm{St}^\alpha_\rho f\ast \varphi_{\rho t}\right) (x). \end{eqnarray*}
\begin{lem}\label{conv2}
Let $f\in\mathcal S'(\mathbb R^d)$, $\varphi\in\mathcal S(\mathbb R^d)$, $\rho$ and $\alpha$ positive real numbers. We have 
 \begin{equation}\mathrm{St}^\alpha_\rho\left[ \left( f\ast\varphi_t\right) \ast P_t\right]  =\left(  \mathrm{St}^\alpha_\rho f \ast\varphi_{\rho t}\right)   \ast P_{\rho t}.\label{poisson1}
 \end{equation}
\end{lem}
Relation (\ref{poisson1}) follows from the fact that 
\begin{eqnarray*}\rho^{\frac{-d}{\alpha}}\left( f \ast \varphi_{\rho^{-1}t} \right)\ast P_{\rho^{-1}t}(\rho^{-1}x)
&=&\rho^{\frac{-d}{\alpha}}\int_{\mathbb R^d}\left( f \ast \varphi_{\rho^{-1}t} \right)(\rho^{-1}x-y) P_{\rho^{-1}t}(y)dy\\
%&=&\rho^{\frac{-d}{\alpha}}\int_{\mathbb R^d}\left\langle f , \varphi_{\rho^{-1}t} (\rho^{-1}x-y-\cdot)\right\rangle  P_{\rho^{-1}t}(y)dy\\
&=&\int_{\mathbb R^d}\left\langle f , \rho^{\frac{d}{\alpha'}}\varphi_{t} (x-z-\rho\cdot)\right\rangle  P_{t}(z)dz\\
&=&\int_{\mathbb R^d}\left\langle \mathrm{St}^\alpha_\rho f , \varphi_{t} (x-z-\cdot)\right\rangle  P_{t}(z)dz\\
&=&\int_{\mathbb R^d} \left( \mathrm{St}^\alpha_\rho f \ast \varphi_{t}\right)  (x-z)  P_{t}(z)dz=\left( \mathrm{St}^\alpha_\rho f \ast \varphi_{t}\right) \ast P_t(x).
\end{eqnarray*}
It comes from Lemma \ref{dilconv} and \ref{conv2} that for a bounded tempered distribution $f$ and $u(x,t)=f\ast P_t(x)$,
\begin{equation}
%\begin{lem}\label{poissonhardy} Let  $f$ be a bounded %tempered distribution. For $\rho>0$ and $\alpha>0$ we have 
\left( \mathrm{St}^\alpha_\rho u\right) (x,t)= \left[  \left( \mathrm{St}^\alpha_\rho f\right) \ast P_t\right] (x),\quad\rho>0\text{  and }\alpha>0\label{convpoisson}
\end{equation}
for all  $t>0$.

\section{Cauchy-Riemann equations, Riesz transforms and  Hardy-Fofana spaces}

Let $u$ be a harmonic function on $\mathbb R^{d+1}_+$; i.e,  $u\in \mathcal{C}^{2}(\mathbb{R}^{d+1}_+)$ and  %$\Delta u= 0$ on $\mathbb{R}^{d},$ where 
 $\Delta u :=\sum_{j=1}^{d+1}\frac{\partial^{2}u}{(\partial x_{j})^{2}}=0,$ where  $x_{d+1}=t$ and  $\mathbb{R}^{d+1}_{+}:=\mathbb{R}^{d}\times ]0,+\infty[$.  We define its non tangential maximal function $u^\ast$ by 
\begin{equation}\label{maxfunct}
u^{\ast}(x):=\sup_{t>0}\sup_{|x-y|<t} |u(y,t)|\;\; \forall x\in\mathbb{R}^{d}.
\end{equation}
Let $f$ be a bounded tempered distribution, and $u(x,t)=P_t\ast f(x)$. %For  a bounded tempered distribution $f$ on $\mathbb R^d$,  $u(x,t):=f\ast P_{t}(x),$  for $(x,t)\in\mathbb{R}^{d+1}_{+}:=\mathbb{R}^{d}\times ]0,+\infty[$ is harmonic on $\mathbb R^{d+1}_+$.  
As we can see in \cite{DmFe}, $u^\ast\in\left( L^p,\ell^q\right)^\alpha(\mathbb R^d)$ whenever $f\in\mathcal H^{(p,q,\alpha)}(\mathbb R^d)$.   We give in the next result a necessary and sufficient conditions for a harmonic function $u$ in $\mathbb R^{d+1}$ to have its non tangential maximal function in $(L^p,\ell^q)^\alpha(\mathbb R^d) $. The proof is based on the dilation characterization of Hardy-Fofana spaces and \cite[Proposition 2.1 ]{AjCf}.
\begin{prop}\label{kady1}
Let $0<p\leq\alpha\leq q<+\infty$ and $u$ an harmonic function on $\mathbb{R}^{d+1}_{+}.$
The maximal function $u^{\ast}$ belongs to $(L^{p},\ell^{q})^{\alpha}(\mathbb R^d)$ if and only if there exists $f\in \mathcal{H}^{(p,q,\alpha)}(\mathbb R^d)$ such that 
$$u(x,t):=f\ast P_{t}(x), \  (x,t)\in\mathbb{R}^{d+1}_{+}.$$
Moreover,
$\left\|f\right\|_{\mathcal{H}^{(p,q,\alpha)}}\approx \left\|u^{\ast}\right\|_{p,q,\alpha}.$

%where the equivalence constants do not depend neither of $ f $ nor of $ u^{\ast} $. 
\end{prop}
\begin{proof} Let $u$ be an harmonic function on $\mathbb{R}^{d+1}_{+},$ and $u^\ast$ the associate non tangential maximal function as defined in Relation (\ref{maxfunct}). 

We suppose that there exists $f\in\mathcal{H}^{(p,q,\alpha)}(\mathbb R^d)$ such that $u(x,t):=f\ast P_{t}(x)$ for all $ (x,t)\in\mathbb{R}^{d+1}_{+}$. 
From the Poisson characterization of Hardy-Fofana spaces (see \cite[Theorem 2.3.8 ]{DmFe}), we deduce that  $\left\|u^{\ast}\right\|_{p,q,\alpha}\leq C\left\|f\right\|_{\mathcal{H}^{(p,q,\alpha)}}$. % thank to  \cite[Theorem 2.3.8 ]{DmFe}.

For the converse, let us suppose that $u^{\ast}\in(L^{p},\ell^{q})^{\alpha}(\mathbb R^d)\subset (L^{p},\ell^{q})(\mathbb R^d)$. It comes from  \cite[Proposition 2.1 ]{AjCf} that there exists $f\in\mathcal H^{(p,q)}(\mathbb R^d)$ and a constant $C>0$ such that 
\begin{equation}
u(x,t)=(f\ast P_t)(x),\qquad (x,t)\in \mathbb R^{d+1}_+
\end{equation}
and 
\begin{equation*}
\frac{1}{C}\Vert f\Vert_{\mathcal H^{(p,q)}}\leq \Vert u^\ast\Vert_{p,q}\leq C\Vert f\Vert_{\mathcal H^{(p,q)}}.
\end{equation*}
Since $\mathrm{St}^\alpha_\rho f\in\mathcal H^{(p,q)}(\mathbb R^d)$ for all $\rho>0$,  $\mathrm{St}^\alpha_\rho u$ harmonic on $\mathbb R^{d+1}_+$ and 
$\left( \mathrm{St}^\alpha_\rho u\right) (x,t)=(\mathrm{St}^\alpha_\rho f)\ast P_t(x)$, it comes that 
$(\mathrm{St}^\alpha_\rho u)^\ast\in (L^p,L^q)(\mathbb R^d)$ and 
\begin{equation*}
\frac{1}{C}\Vert\mathrm{St}^\alpha_\rho  f\Vert_{\mathcal H^{(p,q)}}\leq \Vert (\mathrm{St}^\alpha_\rho  u)^\ast\Vert_{p,q}\leq C\Vert \mathrm{St}^\alpha_\rho  f\Vert_{\mathcal H^{(p,q)}}.
\end{equation*}
This relation being thrue for all $\rho>0$, we have 
\begin{equation*}
\frac{1}{C}\Vert f\Vert_{\mathcal H^{(p,q,\alpha)}}\leq \Vert u^\ast\Vert_{p,q,\alpha}\leq C\Vert f\Vert_{\mathcal H^{(p,q,\alpha)}},
\end{equation*}
where we use the trivial identity 
% \begin{equation*}
$(St^{\alpha}_{\rho} u)^\ast = St^\alpha_\rho u^\ast, \ \rho>0\text{ and } 0<\alpha<\infty.$
\end{proof}

%For the next result, we introduce some class of harmonic vectors functions.

We say that a vector values function  $F:=(u_{1},u_{2},...,u_{d+1})$, with $u_j:\mathbb R^{d+1}_+\rightarrow\mathbb R$,  $j\in\{1,2,...,d+1\}$,  satisfies the generalized Cauchy-Riemann equations (in short $F\in \mathrm{CR}(\mathbb{R}^{d+1}_{+})$) if
\begin{eqnarray}
\frac{\partial u_{j}}{\partial x_{k}}=\frac{\partial u_{k}}{\partial x_{j}}, \;1\leq j,k\leq d+1 \qquad\text{and}\qquad \sum_{j=1}^{d+1}\frac{\partial u_{j}}{\partial x_{j}}=0,
\end{eqnarray}
where we set  $x_{d+1}=t.$  Also recall that  for $j\in\{1,2,...d\},$ the  $j$-th  Riesz transform $\mathcal{R}_{j}(g)$ of a measurable function $g$ is formally defined by
\begin{eqnarray*}
\mathcal{R}_{j}(g)(x):=\lim_{\epsilon\rightarrow 0^+}\int_{|x-y|>\epsilon}K_{j}(x-y)g(y)dy\;\;\;\; \text{a.e}\;\;\; x\in\mathbb{R}^d.
\end{eqnarray*}

where $K_{j}(x):=\frac{\Gamma(\frac{d+1}{2})}{\pi^{\frac{d+1}{2}}}\frac{x_j}{|x|^{d+1}}$,  $x\in \mathbb{R}^d\backslash\{0\}.$

In \cite[Corollary 4.19]{AbFe2}, Ablé and Feuto  demonstrated that Riesz transformations are extendable into bounded linear operators on  Hardy-amalgam spaces $\mathcal H^{(p,q)}(\mathbb R^d)$ for $0<p\leq 1$. We will keep the notations  $\mathcal R_j$, $j=1,\cdots,d$ for these extentions.  Assaubay et al proved the following result.
\begin{prop}[\cite{AjCf}, Proposition 2.3]\label{propA}
Let $\frac{d-1}{d}<\min\left\lbrace p,q\right\rbrace <+\infty$. Suppose that $u$ is harmonic function in $\mathbb{R}^{d+1}_{+}.$ Then $u^{\ast}\in (L^{p},\ell^{q})(\mathbb R^d)$ if and only if there exists an harmonic  vector $F:=(u_{1},...,u_{d+1})\in \mathrm{CR}(\mathbb{R}^{d+1}_{+})$
such that $u_{d+1}:= u$ and $\sup_{t>0}\left\| |F(.,t)|\right\|_{p,q}<+\infty.$ 

Furthermore,  $\sup_{t>0}\left\| |F(.,t)|\right\|_{p,q}\approx \left\|u^{\ast}\right\|_{p,q}$.
\end{prop}
Since  $u^{\ast}\in (L^{p},\ell^{q})(\mathbb R^d)$ if and only if $u=f\ast P_t$ for some $f\in\mathcal H^{(p,q)}(\mathbb R^d)$, they  proved that one can take $u_j(x,t)=\mathcal R_j(f)\ast P_t(x)$, $j=1,\cdots,d$. 

In the case of Hardy-Fofana's spaces, we have the following.

\begin{prop}\label{Kdy}
Assume that $\frac{d-1}{d}<p\leq\alpha\leq q<+\infty$ and $u$ is an harmonic function in $\mathbb{R}^{d+1}_{+}.$ Then $u^{\ast}\in (L^{p},\ell^{q})^{\alpha}(\mathbb R^d)$ if and only if there exists an harmonic  vector $F:=(u_{1},...,u_{d+1})\in \mathrm{CR}(\mathbb{R}^{d+1}_{+})$
such that $u_{d+1}:= u$ and $\sup_{t>0}\left\| |F(.,t)|\right\|_{p,q,\alpha}<+\infty.$ Furthermore
\begin{equation}\sup_{t>0}\left\| |F(.,t)|\right\|_{p,q,\alpha}\approx \left\|u^{\ast}\right\|_{p,q,\alpha}\label{equivn}
\end{equation}
%where the equivalence's constants don't depend on $F$ and $u$ respectively.
\end{prop}

\begin{proof}
Let $\frac{d-1}{d}<p\leq\alpha\leq q<+\infty$ and $u$ an harmonic function on $\mathbb{R}^{d+1}_{+}.$

We suppose that $u^{\ast}\in (L^{p},\ell^{q})^{\alpha}(\mathbb R^d).$ Since $(L^{p},\ell^{q})^{\alpha}(\mathbb R^d)\subset(L^{p},\ell^{q})(\mathbb R^d)$, Proposition \ref{propA} assert that  there exists $f\in  (L^p,\ell^q)(\mathbb R^d)$ so that:
\begin{itemize}
\item  $u(x,t)=f\ast P_t(x)$,
\item  the harmonic vector $F=(u_1,\cdots,u_{d+1})$ with  $u_j(x,t)=R_j(f)\ast P_t(x)$  for $j\in\left\lbrace 1,\cdots,d\right\rbrace $  and $u_{d+1}=u$ belongs to $\mathrm{CR}(\mathbb R^{d+1}_+)$,
\item 
$\sup_{t>0}\Vert\vert F(\cdot,t)\vert\Vert_{p,q}\approx\Vert u^\ast\Vert_{p,q}.$
\end{itemize}
 Since  $u^\ast\in \left( L^p,\ell^q\right) ^\alpha(\mathbb R^d)$ we have that  the tempered distribution $f$ belongs to $\mathcal{H}^{(p,q,\alpha)}(\mathbb R^d)$, thanks to Proposition \ref{kady1}.  All we have to prove now is that $x\mapsto F(x,t)$ belongs to $(L^{p},\ell^{q})^{\alpha}(\mathbb R^d)$ for $t>0$ and that  Relation (\ref{equivn}) is satisfies.

Fix $t>0$ and $\rho>0$.  %We have it is easy to see that $\left( \mathrm{St}^\alpha_\rho u\right) ^\ast=\mathrm{St}^\alpha_\rho \left( u ^\ast\right)$.
 Since $u^\ast\in \left( L^p,\ell^q\right) ^\alpha(\mathbb R^d)$ and $\left( \mathrm{St}^\alpha_\rho u\right) ^\ast=\mathrm{St}^\alpha_\rho \left( u ^\ast\right)$, we have that for $\rho>0$, 
$\Vert \left(\mathrm{St}^\alpha_\rho  u\right)^\ast\Vert_{p,q}\leq \Vert  u^\ast\Vert_{p,q,\alpha}$. Hence $\left( \mathrm{St}^\alpha_\rho u\right)^\ast\in\left( L^p,\ell^q\right) (\mathbb R^d)$ so that there exists $f_\rho\in\mathcal H^{(p,q)}(\mathbb R^d)$ satisfying 
$$(\mathrm{St}^\alpha_\rho u)(x,t)=(f_\rho\ast P_t)(x),$$
with \begin{equation}
F_\rho(x,t):=\left( \mathcal R_1(f_\rho)\ast P_t(x),\cdots,\mathcal R_d(f^\rho)\ast P_t(x),  (f_\rho)\ast P_t)(x)\right)\label{vectordilation}
\end{equation}
 belonging to $\mathrm{CR}_+(\mathbb R^{d+1}_+)$ and 
\begin{equation}\sup_{t>0}\Vert\vert F_\rho(\cdot,t)\vert\Vert_{p,q}\approx\Vert \mathrm{St}^\alpha_\rho (u^\ast)\Vert_{p,q}.\label{equiv}
\end{equation}
%But, $u(x,t)=f\ast P_t(x)$ so that 
Moreover $(\mathrm{St}^\alpha_\rho u)(x,t)=(\mathrm{St}^\alpha_\rho f)\ast P_t(x)$, thanks to Relation (\ref{convpoisson}). It follows that $f_\rho\ast P_t=(\mathrm{St}^\alpha_\rho f)\ast P_t$ for all $t>0$ so that $f_\rho=\mathrm{St}^\alpha_\rho f$. We recall that the last equality comes from the fact that for $f\in \mathcal H^{(p,q)}(\mathbb R^d)$, $f\ast P_t$ tends to $f$ in $\mathcal S'(\mathbb R^d)$ as $t$ goes to 0. Replacing  $f_\rho$ by $\mathrm{St}^\alpha_\rho f$  in Relation (\ref{vectordilation}) yields  $F_\rho(x,t)=\left( \mathcal R_1(\mathrm{St}^\alpha_\rho f)\ast P_t(x),\cdots,\mathcal R_d(\mathrm{St}^\alpha_\rho f)\ast P_t(x),  (\mathrm{St}^\alpha_\rho f)\ast P_t)(x)\right)$. Since the operator $\mathrm{St}^\alpha_\rho $ commute with $R_j$ we have that
\begin{eqnarray*}F_\rho(\cdot,t)&=&\left( \mathrm{St}^\alpha_\rho (\mathcal R_1f)\ast P_t(\cdot),\cdots,\mathrm{St}^\alpha_\rho (\mathcal R_d f)\ast P_t(\cdot),  (\mathrm{St}^\alpha_\rho f)\ast P_t)(\cdot)\right)\\
&=&\left( \mathrm{St}^\alpha_\rho \left( u_1(\cdot,\rho^{-1}t)\right) ,\cdots,\mathrm{St}^\alpha_\rho\left(  u_d(\cdot,\rho^{-1}t)\right) ,\mathrm{St}^\alpha_\rho \left( u_{d+1}(\cdot,\rho^{-1}t)\right) \right) \\
&=&\mathrm{St}^\alpha_\rho \left( F(\cdot,\rho^{-1}t)\right). 
\end{eqnarray*}
If we take this expression of $F_\rho$ in Relation (\ref{equiv}) we obtain that 
$$\sup_{t>0}\Vert\vert \mathrm{St}^\alpha_\rho (F(\cdot,\rho^{-1}t)\vert\Vert_{p,q}\approx\Vert \mathrm{St}^\alpha_\rho (u^\ast)\Vert_{p,q}.$$
But $\sup_{t>0}\Vert\vert \mathrm{St}^\alpha_\rho (F(\cdot,\rho^{-1}t))\vert\Vert_{p,q}=\sup_{t>0}\Vert\vert \mathrm{St}^\alpha_\rho (F(\cdot,t))\vert\Vert_{p,q}$ and the result follow from the definition of Hardy-Fofana space.
\end{proof}

The next result gives a characterization of $\mathcal H^{(p,q,\alpha)}(\mathbb R^d)$ via Riesz transforms $\mathcal R_j(f\ast \phi)$. Since we need to use the characterization of $\mathcal H^{(p,q)}(\mathbb R^d)$ given in \cite{AjCf},  we give the following definition.

\begin{defn} Let $0<p\leq \alpha\leq q<\infty$.
%Following the definition given by Assaubay et al. in \cite{AjCf}, we say that 
A tempered distribution $f$ is said to be :
\begin{itemize}
\item $(p,q)$-restricted at infinity if there exists $\mu_{0}\geq 1$ such that for $\mu\geq \mu_{0}$, we have  $$f\ast \phi\in (L^{p\mu},\ell^{q\mu})(\mathbb R^d),\quad\phi \in \mathcal{S}(\mathbb{R}^{d}).$$ 
\item 
$(p,q,\alpha)$-restricted at infinity if there exists $\mu_{0}\geq 1$ such that for $\mu\geq \mu_{0}$, we have  $$f\ast \phi\in (L^{p\mu},\ell^{q\mu})^{\alpha \mu}(\mathbb R^d),\quad\phi \in \mathcal{S}(\mathbb{R}^{d}).$$ 
\end{itemize}
%$(p,q,\alpha)$-restricted at infinity if there exists $\mu_{0}\geq 1$ such that for $\mu\geq \mu_{0}$, we have  $$f\ast \phi\in (L^{p\mu},\ell^{q\mu})^{\alpha \mu},\qquad\phi \in \mathcal{S}(\mathbb{R}^{d}).$$ 
\end{defn}
It is easy to see that tempered distributions which are $(p,q,\alpha)$-restricted for $p\leq \alpha\leq q$, are also $(p,q)$-restricted. 
Theorem 1.1 in \cite{AjCf} assert that  a tempered distribution $f$ belongs to $\mathcal H^{(p,q)}(\mathbb R^d)$ for  $\frac{d-1}{d}<\min{(p,q)}<\infty$, if and only if it is $(p,q)$-restricted at infty and, for $\phi\in\mathcal S(\mathbb R^d)$ with non vanish integral,
$$\sup_{t>0}\left( \Vert f\ast \phi_t\Vert_{p,q}+\sum_{j=1}^{d}\Vert (\mathcal R_jf)\ast\phi_t\Vert_{p,q}\right) <\infty.$$
When this is the case, 
$$\Vert f\Vert_{\mathcal H^{(p,q)}}\approx \sup_{t>0}\left( \Vert f\ast \phi_t\Vert_{p,q}+\sum_{j=1}^{d}\Vert (\mathcal R_jf)\ast\phi_t\Vert_{p,q}\right) .$$

In the case of Hardy-Fofana space, we have the following result.
\begin{thm}\label{yvo}
Let  $\frac{d-1}{d}<p\leq\alpha\leq q<\infty$, $f\in\mathcal{S'}(\mathbb{R}^{d})$.%and $\phi\in\mathcal{S}(\mathbb{R}^{d})$ such that $\int_{\mathbb{R}^{d}}\phi(x)dx\neq 0.$
%\begin{enumerate}
%\item
% Let $p>\frac{d-1}{d}.$ 
  Then $f\in\mathcal{H}^{(p,q,\alpha)}(\mathbb R^d)$ if and only if $f$ is $(p,q,\alpha)$-restricted at infinity and, for $\phi\in\mathcal S(\mathbb R^d)$ with non vanish integral,
\begin{equation}
\sup_{t>0}\left( \left\|f\ast \phi_{t}\right\|_{p,q,\alpha}+ \sum_{j=1}^{d} \left\|(\mathcal R_{j}f)\ast\phi_{t}\right\|_{p,q,\alpha} \right) <+\infty.\label{cond2}
\end{equation}
Moreover, % with
\begin{equation}
 \left\|f\right\|_{\mathcal{H}^{(p,q,\alpha)}}  \approx \sup_{t>0}\left( \left\|f\ast \phi_{t}\right\|_{p,q,\alpha}+ \sum_{j=1}^{d} \left\|(\mathcal R_{j}f)\ast\phi_{t}\right\|_{p,q,\alpha} \right) .\label{equiv2}
 \end{equation}

\end{thm}
\begin{proof}
%Let's prove the first point.
%$\bullet$ 
Let $\frac{d-1}{d}<p\leq\alpha\leq q<\infty$ and  $f\in\mathcal{S'}(\mathbb{R}^{d})$.

We suppose that $f$ is $(p,q,\alpha)$-restricted at infinity and satisfies (\ref{cond2}) for non vanishing Schwartz function $\phi$. There exists $\mu_0>1$ (large enought) such that for $\mu>\mu_0$, we have 
\begin{equation}
f\ast\phi\in \left( L^{p\mu};\ell^{q\mu}\right) ^{\alpha\mu}(\mathbb R^d),\quad \phi\in\mathcal S(\mathbb R^d).
\end{equation}
%Fixed $\varphi\in \mathcal S(\mathbb R^d)$. 
It comes from the definition of Fofana spaces that 
$$\mathrm{St}^{\alpha\mu}_\rho\left( f\ast\phi\right) \in \left( L^{p\mu},\ell^{\mu q}\right) (\mathbb R^d)\quad \phi\in\mathcal S(\mathbb R^d), \quad \rho>0.$$
Taking $\rho=1$, we obtain that  $f$ is $(p,q)$-restricted at infinity. Since for all $\phi\in\mathcal S(\mathbb R^d)$ with non vanishing integral we also have that 
$$A=\sup_{t>0}\left( \sup_{\rho}\left\|\mathrm{St}^\alpha_\rho\left( f\ast \phi_{t}\right) \right\|_{p,q}+ \sum_{j=1}^{d} \sup_{\rho>0}\left\|\mathrm{St}^\alpha_\rho\left( (\mathcal R_{j}f)\ast\phi_{t}\right) \right\|_{p,q} \right) <\infty,$$
it follows that $$\sup_{t>0}\left(\left\| f\ast \phi_{t} \right\|_{p,q}+ \sum_{j=1}^{d} \left\| (\mathcal R_{j}f)\ast\phi_{t} \right\|_{p,q} \right) \leq A.$$
 Thus $f\in\mathcal H^{(p,q)}(\mathbb R^d)$ thanks to  \cite[Theorem 1.1]{AjCf}. It remains to prove that  the familly $\left\lbrace \mathrm{St}^\alpha_\rho f\right\rbrace _{\rho>0}$ is uniformly bounded in $\mathcal H^{(p,q)}(\mathbb R^d)$.

Fix $\rho>0$. We have $ \mathrm{St}^\alpha_\rho f\in  \mathcal H^{(p,q)}(\mathbb R^d)$ so that 
$$\Vert \mathrm{St}^\alpha_\rho f\Vert_{\mathcal H^{(p,q)}}\approx\sup_{t>0}\left( \left\|\mathrm{St}^\alpha_\rho( f)\ast \phi_{t} \right\|_{p,q}+ \sum_{j=1}^{d} \left\|\mathcal R_{j}(\mathrm{St}^\alpha_\rho f)\ast\phi_{t} \right\|_{p,q} \right)$$
thanks once more to  \cite[Theorem 1.1]{AjCf}. But we have in one hand that

 $\mathcal R_j(f)\ast\phi_t=\mathcal R_j\left( f\ast\phi_t\right) $, so that 
\begin{equation}
\mathrm{St}^\alpha_\rho\left[\left(  \mathcal R_jf\right) \ast\phi_t\right] =\mathrm{St}^\alpha_\rho\left[ \mathcal R_j\left( f \ast\phi_t\right)\right]  =\mathcal R_j\left[ \mathrm{St}^\alpha_\rho\left( f \ast\phi_t\right)\right],\label{aa}
\end{equation}
 %\label{aa}
%\end{equation} 
where the last equality comes from the fact that dilation comute with Riesz transforms. In the other hand   we have that \begin{equation}
\sup_{t>0}\Vert \mathrm{St}^\alpha_\rho\left( f \ast\phi_t\right)\Vert_{p,q}=\sup_{t>0}\Vert \mathrm{St}^\alpha_\rho\left( f \right) \ast\phi_t\Vert_{p,q},\label{bb}
\end{equation}
  thanks to Lemma \ref{dilconv}.  Therefore, we have 
$$\sup_{\rho>0}\sup_{t>0}\Vert \mathrm{St}^\alpha_\rho\left( f \right) \ast\phi_t\Vert_{p,q}=\sup_{\rho>0}\sup_{t>0}\Vert \mathrm{St}^\alpha_\rho\left( f \ast\phi_t\right)\Vert_{p,q}\leq A$$
 and 
 $$\sup_{\rho>0}\sup_{t>0}\sum_{j=1}^d\Vert \mathcal R_j\left[ \mathrm{St}^\alpha_\rho\left( f \ast\phi_t\right)\right]\Vert_{p,q}=\sup_{\rho>0}\sup_{t>0}\sum_{j=1}^d\Vert \mathrm{St}^\alpha_\rho\left[\left(  \mathcal R_jf\right) \ast\phi_t\right] \Vert_{p,q}\leq A.$$
We deduce that $\sup_{\rho>0}\Vert \mathrm{St}^\alpha_\rho f\Vert_{\mathcal H^{(p,q)}}<\infty$, which prove  that  $f\in\mathcal H^{(p,q,\alpha)}(\mathbb R^d)$. 

For the converse, 
we suppose that $f\in\mathcal H^{(p,q,\alpha)}(\mathbb R^d)$. It follows that $\mathrm{St}^\alpha_\rho f\in \mathcal H^{(p,q)}(\mathbb R^d)$ with $\Vert \mathrm{St}^{\alpha}_\rho f\Vert_{\mathcal H^{(p,q)}}\leq\Vert f\Vert_{\mathcal H^{(p,q,\alpha)}}<\infty$ for all $\rho>0$.  It comes from  \cite[Theorem 1.1]{AjCf} that $\mathrm{St}^{\alpha}_\rho f$ is $(p,q)$-resticted at infinity and 
$$\Vert \mathrm{St}^\alpha_\rho f\Vert_{\mathcal H^{(p,q)}}\approx\sup_{t>0}\left( \left\|\mathrm{St}^\alpha_\rho( f)\ast \phi_{t} \right\|_{p,q}+ \sum_{j=1}^{d} \left\|\mathrm{St}^\alpha_\rho ( \mathcal R_{j}f)\ast\phi_{t} \right\|_{p,q} \right)$$
for all $\phi\in\mathcal S(\mathbb R^d)$ with non vanish integral. 

From Relations (\ref{bb})  and (\ref{aa}), and the definitions of $\Vert\cdot\Vert_{p,q,\alpha}$ and of $\Vert\cdot\Vert_{\mathcal H^{(p,q,\alpha)}}$,  we have that 
$$ \left\|f\right\|_{\mathcal{H}^{(p,q,\alpha)}}  \approx \sup_{t>0}\bigg(\left\|f\ast \phi_{t}\right\|_{p,q,\alpha}+ \sum_{j=1}^{d} \left\|(\mathcal R_{j}f)\ast\phi_{t}\right\|_{p,q,\alpha} \bigg)<\infty.$$

%Since $f\in \mathcal{H}^{(q,p,\alpha)}.$
 Let $\phi\in \mathcal{S}(\mathbb{R}^{d})$.  We have $\left\|f\ast \phi\right\|_{p,q,\alpha}\leq C\left\|f\right\|_{\mathcal{H}^{(p,q,\alpha)}}$. For $\mu\geq 1$  we have $f\ast \phi \in (L^{p\mu}, \ell^{q \mu} )^{\alpha \mu}$. In fact assuming that $\Vert f\ast \varphi\Vert_{\infty}\neq 0$  we have 
\begin{equation*}
f\ast \phi \in (L^{p}, \ell^{q} )^{\alpha}\; \; \text{and}\;\;
\left\|f\ast \phi\right\|_{p\mu,q\mu,\alpha\mu}\leq C \left\|f\ast \phi\right\|_{\infty}^{1-\frac{1}{\mu}} \left\|f\ast \phi\right\|_{p,q,\alpha}^{\frac{1}{\mu}}
\end{equation*}
and then $f$ is $(q,p,\alpha)$-restricted at infinity.
\end{proof}

\section{Temperature Cauchy-Riemann equations and Hardy-Fofana spaces}

A vector  $ F = (u_1, u_2, \cdots, u_ {d + 1}) $ of functions in $\mathbb R^{d+1}_+$ satisfy the generalized temperature Cauchy-Riemann equations, if it  satisfies the following conditions :
\begin{enumerate}%[i]
%\item the $u_j$'s satisfy
%\begin{equation}
%\left\lbrace \begin{array}{l}
\item $\sum_{j=1}^{d}\frac{\partial u_j}{\partial x_j}=i\partial^{1/2}_t u_{d+1}$
\item  $\frac{\partial u_j}{\partial x_k}=\frac{\partial u_k}{\partial x_j}\text{ for } j,k=1,2,\cdots,d$
\item $\frac{\partial u_{d+1}}{\partial x_j}=-i \partial^{1/2}_t u_j, j=1,2,\cdots,d$,   %\end{array}\right.,$ %\label{TCR}
%\end{equation}
        with 
$$(\partial^{1/2}_t g)(t):=\frac{e^{i\pi/2}}{\sqrt{\pi}}\int_t^\infty\frac{g'(s)}{\sqrt{s-t}}ds, t>0$$
when $g$ is a smooth enough function on $\left( 0,\infty\right) $ 
\end{enumerate}

In \cite{AjCf}, the authors defined the space $\mathbb H^{p,q}(\mathbb R^{d+1}_+)$ ($0< p,q<\infty$) as the vector space of
 vector functions $ F = (u_1, u_2, \cdots, u_ {d + 1}) $ satisfying
  generalized temperature Cauchy-Riemann equations and such that 
$$\Vert F\Vert_{\mathbb H^{p,q}(\mathbb R^{d+1}_+)}:=\sup_{t>0}\Vert \vert F(\cdot,t)\vert\Vert_{p,q}<\infty.$$
%\end{enumerate}
They also proved that under appropriate conditions on the exponents $p$ and $q$, the space $\mathbb H^{p,q}(\mathbb R^{d+1}_+)$ is topologically isomorphic to $\mathcal H^{p,q}(\mathbb R^d)$. To carry out the proof of this result,  they use a subspace of what they call the temperature space $\mathcal T(\mathbb R^{d+1}_+)$; that is the space of functions $u\in\mathcal C^2(\mathbb R^{d+1}_+)$, satisfying 
$$\frac{\partial u}{\partial t}=\sum_{j=1}^{d}\frac{\partial^{2}u}{\partial x_{j}^{2}}\text{ in }\mathbb{R}_{+}^{d+1}.$$
More precisely, for  $0<p,q<\infty$, they put 
  $$\mathcal T^{p,q}(\mathbb R^{d+1}_+):=\left\lbrace u\in\mathcal T(\mathbb R^{d+1}_+):\vert\vert u\vert\vert_{\mathcal{T}^{(p,q)}}< \infty\right\rbrace $$
  where 
  $$\vert\vert u\vert\vert_{\mathcal{T}^{(p,q)}}:=\sup_{t>0}\vert\vert u(.,t)\vert\vert_{q,p}.$$
They proved \cite[Proposition 3.2 (i)]{AjCf}  that for $\frac{d-1}{d}<p,q<\infty$, $F=(u_1,u_2,\cdots,u_{d+1})\in \mathbb H^{p,q}(\mathbb R^{d+1}_+)$ implies that $u:=u_{d+1}\in\mathcal T^{p,q}(\mathbb R^{d+1}_+)$ and $u_j(\cdot,t)=\mathcal R_j(u(\cdot, t)),\quad t>0,\quad j=1,\cdots,d$. 

We claim that for $0<p\leq \alpha\leq q<\infty$ and $r>0$, the space $\mathcal T^{p,q}(\mathbb R^{d+1}_+)$ is stable under the dilation $\mathrm{St}^\alpha_r$. This is due to the fact that   for $f\in \left( L^p,\ell^q\right) (\mathbb R^d)$, there exists a constant $C(\alpha,r,p,q)>0$ such that 
$$C(\alpha,r,p,q)^{-1}\Vert f\Vert_{p,q}\leq \Vert \mathrm{St}^\alpha_r f\Vert_{p,q}\leq C(\alpha,r,p,q)\Vert f\Vert_{p,q},$$
and this dilation commute with Riesz transforms. It follows that if $F=(u_1,\cdots,u_{d+1})\in\mathbb H^{p,q}(\mathbb R^{d+1}_+)$ then  $\mathrm{St}^\alpha_r F\in\mathbb H^{p,q}(\mathbb R^{d+1}_+)$. %if the space $\mathcal T^{p,q}(\mathbb R^{d+1}_+)$ is stable under the dilation $\mathrm{St}^\alpha_r $, thanks to the commutation of dilations $\mathrm{St}^\alpha_r F$ with Riesz transforms. %But it is well known that for $f\in \left( L^p,\ell^q\right) (\mathbb R^d)$, there exists a constant $C(\alpha,r,p,q)>0$ such that 
%$$C(\alpha,r,p,q)^{-1}\Vert f\Vert_{p,q}\leq \Vert \mathrm{St}^\alpha_r f\Vert_{p,q}\leq C(\alpha,r,p,q)\Vert f\Vert_{p,q}.$$
%Thus the space $\mathcal T^{p,q}$ is stable under $\mathrm{St}^\alpha_r$. 

We put 
$$\Vert F\Vert_{\mathbb H^{(p,q,\alpha)}}:=\sup_{r>0}\Vert \mathrm{St}^\alpha_r F\Vert_{\mathbb H^{p,q}(\mathbb R^{d+1}_+)}$$
and defined the space  $\mathbb H^{(p,q,\alpha)}(\mathbb R^{d+1}_+)$ as the subspace of $\mathbb H^{p,q}(\mathbb R^{d+1}_+)$ consits of $F$ satisfying $\Vert F\Vert_{\mathbb H^{(p,q,\alpha)}}<\infty$. We have the following result in  Hardy-Fofana spaces. 

\begin{thm}\label{yvo2}Let $\frac{d-1}{d} < p\leq \alpha \leq q < \infty,$ and $W_t$ the heat kernel defined by 
$$W_t(x)=\frac{e^{-\vert x\vert^2/4t}}{(4\pi t)^{d/2}}.$$
 The map $\mathcal L$ define on $\mathcal{H}^{(p,q,\alpha)}(\mathbb{R}^{d})$ by %We define for every tempered distribution $f\in \mathcal{H}^{(p,q,\alpha)},$ the application $
$$\mathcal{L}(f)(x,t):=\big(((\mathcal{R}_{1}f)\ast W_{t})(x),\cdots,((\mathcal{R}_{d}f)\ast W_{t})(x),(f\ast W_{t})(x)\big)$$for all $x\in\mathbb{R}^{d}$ and $t>0$, is a topological isomorphism from $\mathcal{H}^{(p,q,\alpha)}(\mathbb{R}^{d})$ onto $\mathbb{H}^{(p,q,\alpha)}(\mathbb{R}_{+}^{d+1}).$  \end{thm} 

\begin{proof}Let $f\in\mathcal{H}^{(p,q,\alpha)}(\mathbb R^d)$.  For $r>0$ we have $\mathrm{St}^\alpha_rf\in  \mathcal H^{p,q}(\mathbb R^d),$ thanks to the definition of  $\mathcal{H}^{(p,q,\alpha)}(\mathbb R^d)$.  It comes from \cite[Theorem 1.3]{AjCf} that 
\begin{equation}\mathcal{L}(St^\alpha_r f)\in\mathbb{H}^{p,q}(\mathbb{R}_{+}^{d+1}),\text{  with  }
\Vert \mathcal L(\mathrm{St}^\alpha_r f)\Vert_{\mathbb H^{p,q}(\mathbb R^{d+1}_+)}\leq C \Vert \mathrm{St}^\alpha_r f\Vert _{\mathcal H^{p,q}(\mathbb R^d)}\label{cont1}
\end{equation}
for all $r>0$. Since $\mathcal{R}_{j}(f\ast W_{t})=\mathcal{R}_{j}(f)\ast W_{t}$ and $\mathrm{St}^\alpha_r$ commute with Riesz transforms, we have that $\mathcal{L}(St^\alpha_r f)=St^\alpha_r(\mathcal{L}( f))$, $r>0$. Thaking this remark in (\ref{cont1}) we obtain that 
$$\Vert \mathcal L( f)\Vert_{\mathbb H^{(p,q,\alpha)}(\mathbb R^{d+1}_+)}\leq C \Vert  f\Vert _{\mathcal H^{(p,q,\alpha)}(\mathbb R^d)}.$$

Let now $F=(u_1, u_2,\cdots,u_d,u_{d+1})$ belonging to  $\mathbb H^{(p,q,\alpha)}(\mathbb R^{d+1}_+)$. This implies that  $\mathrm{St}^\alpha_rF=(\mathrm{St}^\alpha_r u_1, \mathrm{St}^\alpha_r u_2,\cdots,\mathrm{St}^\alpha_r u_d,\mathrm{St}^\alpha_r u_{d+1})\in\mathbb H^{p,q}(\mathbb R^{d+1}_+)$ for all $r>0$. As we can see in the proof of \cite[Theorem 3.1]{AjCf} this implies that for all $r>0$, there exists $f_r\in\mathcal H^{p,q}(\mathbb R^d)$ so that 
$\mathrm{St}^\alpha_r u_{d+1}\in\mathcal T^{p,q}(\mathbb R^{d+1}_+)$  with $\mathrm{St}^\alpha_r u_{d+1}(x,t)=f_r\ast W_t(x)$ and 
\begin{equation}\Vert  f_r\Vert_{\mathcal H^{p,q}}\leq C\sup_{t>0}\Vert \mathrm{St }^\alpha_rF(\cdot,t)\Vert_{p,q}\leq C\Vert F\Vert_{\mathbb H^{(p,q,\alpha)}},\label{contr}
%sup_{t>0}\Vert \mathrm{St }^\alpha_rF(\cdot,t)\Vert_{p,q},
\end{equation}
and $\mathrm{St}^\alpha_r u_j(\cdot,t)=\mathcal R_j(\mathrm{St}^\alpha_r u_{d+1}(\cdot,t))$, $t>0$, $j=1,\cdots,d.$

We put $f:=f^1$. %We are going to prove that for $r>0$ we have $f_r=\mathrm{St}^\alpha_rf$.
We have $\mathrm{St }^\alpha_r f=f_r$ for all  $r>0$. Taking this in estimate (\ref{contr}) yields 
$$\Vert \mathrm{St }^\alpha_r f\Vert_{\mathcal H^{p,q}}\leq C\Vert F\Vert_{\mathbb H^{(p,q,\alpha)}}$$
wich prove that $f\in \mathcal H^{(p,q,\alpha)}(\mathbb R^d)$.

The vector $$G(x,t)=((\mathcal R_1(f)\ast P_t)(x),\cdots,\mathcal R_d(f)\ast P_t)(x),(f\ast P_t)(x)),\quad x\in\mathbb R^d,t>0$$
is harmonic, satisfies the generalized Cauchy-Riemann equation, and 
$$\sup_{t>0}\Vert \vert G(\cdot,t)\vert\Vert_{p,q,\alpha}\leq C\Vert F\Vert_{\mathbb H^{(p,q,\alpha)}}.$$

\end{proof}

\end{document}